\documentclass[10pt,leqno]{article}
\usepackage{amsmath}
\usepackage{amsthm}
\usepackage{amssymb}
\usepackage{amscd}
\usepackage{graphicx}
\usepackage{epsfig}

\numberwithin{equation}{section}
\newtheorem{theorem}{Theorem}[section]

\newtheorem{proposition}[theorem]{Proposition}
\newtheorem{corollary}[theorem]{Corollary}

\theoremstyle{definition}
\newtheorem{definition}[theorem]{Definition}

\newtheorem{remark}[theorem]{Remark}

\numberwithin{equation}{section}

\usepackage[all]{xy}

\newcommand{\st}{\;|\;}

\newcommand{\R}{\mathbb{R}}

\newcommand{\Z}{\mathbb{Z}}
\newcommand{\C}{\mathbb{C}}

\newcommand{\VV}{\mathbb{V}}

\newcommand{\M}{\mathcal{M}}
\newcommand{\N}{\mathcal{N}}

\newcommand{\cR}{\mathcal{R}}
\newcommand{\cS}{\mathcal{S}}

\newcommand{\PSL}{\mathrm{PSL}}

\newcommand{\Sp}{\mathrm{Sp}}

\newcommand{\U}{\mathrm{U}}

\newcommand{\GL}{\mathrm{GL}}
\newcommand{\SL}{\mathrm{SL}}

\newcommand{\SO}{\mathrm{SO}}
\newcommand{\cSp}{\mathrm{Sp}}

\DeclareMathOperator{\Jac}{Jac} 

\DeclareMathOperator{\Ad}{Ad} 

\DeclareMathOperator{\divisor}{div}

\DeclareMathOperator{\cSym}{Sym}

\DeclareMathOperator{\tr}{tr}
\DeclareMathOperator{\rk}{rk}

\DeclareMathOperator{\End}{End}

\newcommand{\liez}{\mathfrak{z}}

\newcommand{\liemc}{\mathfrak{m}^{\mathbb{C}}}
\newcommand{\lieh}{\mathfrak{h}}
\newcommand{\liehc}{\mathfrak{h}^{\mathbb{C}}}

\newcommand{\liegc}{\mathfrak{g}^{\mathbb{C}}}

\allowdisplaybreaks
\begin{document}
\thispagestyle{empty}




\markboth{Andr\'e Oliveira}{Quadric bundles applied to non-maximal Higgs bundles}
$ $
\bigskip

\bigskip

\centerline{{\LARGE  Quadric bundles applied to non-maximal Higgs bundles}}

\bigskip
\bigskip
\centerline{{\large  Andr\'e Oliveira}}

\vspace*{.7cm}

\begin{abstract}
We present a survey on the moduli spaces of rank $2$ quadric bundles over a compact Riemann surface $X$. These are objects which generalise orthogonal bundles and which naturally occur through the study of the connected components of the moduli spaces of Higgs bundles over $X$ for the real symplectic group $\Sp(4,\R)$, with non-maximal Toledo invariant. Hence they are also related with the moduli space of representations of $\pi_1(X)$ in $\Sp(4,\R)$. We explain this motivation in some detail.
\end{abstract}

\pagestyle{myheadings}
\section{Components of Higgs bundles moduli spaces}\label{sec:1}

Higgs bundles over a compact Riemann surface $X$ were introduced by Nigel Hitchin in \cite{hitchin:1987} as a pair $(V,\varphi)$ where $V$ is a rank $2$ and degree $d$ holomorphic vector bundle on $X$, with fixed determinant, and $\varphi$ a section of $\End(V)\otimes K$ with trace zero. $K$ denotes the canonical line bundle of $X$ --- the cotangent bundle of $X$. Nowadays those are also known as $\SL(2,\C)$-Higgs bundles. In the same paper, Hitchin determined the Poincar\'e polynomial of the corresponding moduli space $\M_d(\SL(2,\C))$, for $d$ odd. The method was based on Morse-Bott theory, so smoothness of the moduli was an essential feature. It was then clear that $\M_d(\SL(2,\C))$ has an extremely rich topological structure, so a natural question was to ask about the topology of the moduli spaces of Higgs bundles for other groups. For $\SL(3,\C)$, this was achieved by P. Gothen in \cite{gothen:1994} and more recently, and using a new approach, O. Garc\'\i a-Prada, J. Heinloth and A. Schmitt in \cite{garcia-prada-heinloth-schmitt:2014,garcia-prada-heinloth:2013} obtained the same for $\SL(4,\C)$ and recursive formulas for $\SL(n,\C)$. Other recent developments were achieved in \cite{schiffmann:2016} on the study of $\M_d(\SL(n,\C))$, which seem to confirm some fascinating conjectures \cite{hausel-rodriguez-villegas:2008}. All these cases were done under the condition of coprimality between rank and degree, so that the moduli spaces are smooth. 

However, for a general real, connected, semisimple Lie group $G$, the moduli spaces $\M_c(G)$ of $G$-Higgs bundles with fixed topological type $c\in\pi_1(G)$, are non-smooth. This is one of the reasons why the topology of $\M_c(G)$ is basically unknown. Still, their most basic topological invariant --- the number of connected components --- is a honourable exception in this unknown territory, and much is known about it. If $G$ is compact then $\M_c(G)$ is non-empty and connected for any $c\in\pi_1(G)$ \cite{ramanathan:1975} and the same is true if $G$ is complex \cite{garcia-prada-oliveira:2014,li:1993}. In both cases the same holds even if $G$ is just reductive or even non-connected (the only difference is that for non-connected groups, the topological type is indexed not by $\pi_1(G)$, but by a different set \cite{oliveira:2011}). When $G$ is a real group, the situation can be drastically different. There are two cases where extra components are known to occur: when $G$ is a split real form of $G^\C$ and when $G$ is a non-compact group of hermitian type.

Suppose $G$ is a split real form of $G^\C$. Intuitively this means that $G$ is the ``maximally non-compact'' real form of $G^\C$; see for example \cite{onishchik:2004} for the precise definition. For instance $\SL(n,\R)$ and $\Sp(2n,\R)$  are split real forms of $\SL(n,\C)$ and of $\Sp(2n,\C)$, respectively. For these groups, Hitchin proved in \cite{hitchin:1992} that there always exists at least one topological type $c$ for which $\M_c(G)$ is disconnected and that the ``extra'' component is contractible and indeed isomorphic to a vector space --- this is the celebrated \emph{Hitchin component} also known as \emph{Teichm\"uller component}. We will not pursue in this direction here.

A non-compact semisimple Lie group $G$ of \emph{hermitian type} is characterised by the fact that $G/H$ is a hermitian symmetric space, where $H\subset G$ is a maximal compact subgroup. Thus $G/H$ admits a complex structure compatible with the Riemannian structure, making it a K\"ahler manifold. If $G/H$ is irreducible, the centre of the Lie algebra of $H$ is one-dimensional and this implies that the torsion-free part of $\pi_1(G)=\pi_1(H)$ is isomorphic to $\Z$, hence the topological type gives rise to an integer $d$ (usually the degree of some vector bundle), called the \emph{Toledo invariant}. This Toledo invariant is subject to a bound condition, called the \emph{Milnor-Wood inequality}, beyond which the moduli spaces $\M_d(G)$ are empty. Moreover, when $|d|$ is maximal (and $G$ is of tube type \cite{bradlow-garcia-prada-gothen:2005}) there is a so-called \emph{Cayler partner phenomena} which implies the existence of extra components for $\M_d(G)$. This has been studied for many classes of hermitian type groups \cite{bradlow-garcia-prada-gothen:2005} and proved in an intrinsic and general way recently in \cite{biquard-garcia-prada-rubio:2015}.

On the other hand, the connected components of $\M_d(G)$ for non-maximal and non-zero Toledo invariant are not known in general. One exception is the case of $\U(p,q)$, which has been basically dealt in \cite{bradlow-garcia-prada-gothen:2003,bradlow-garcia-prada-gothen:2004 triples}. Two other exceptions are the cases of $G=\Sp(4,\R)$ and of $G=\SO_0(2,3)$ --- the identity component of $\SO(2,3)$. In these two cases, it is known \cite{garcia-prada-mundet:2004,gothen-oliveira:2012} that all the non-maximal subspaces are connected for each fixed topological type. Note that in the case of $\SO_0(2,3)$, the topological type is given by an element $(d,w)\in\Z\times\Z/2=\pi_1(\SO_0(2,3))$, with $d$ being the Toledo invariant; so for each $d$ there are two components, labeled by $w$. We expect that the same holds true in general, that is, $\M_d(G)$ is connected for non-maximal $d$ and fixed topological type.

In this paper we give an overview of the proof given in \cite{gothen-oliveira:2012} of the connectedness of $\M_d(\Sp(4,\R))$ and $\M_d(\SO_0(2,3))$ for non-maximal and non-zero $d$. In this study one is naturally lead to consider a certain type of pairs, which we call \emph{quadric bundles}, and the corresponding moduli spaces, depending on a real parameter $\alpha$. Denote them by $\N_\alpha(d)$. The relevant parameter for the study of $\M_d(\Sp(4,\R))$ and $\M_d(\SO_0(2,3))$ is $\alpha=0$. The idea is to obtain a description of the connected components of $\N_{\alpha_m^-}(d)$, for a specific value $\alpha_m^-$ of the parameter $\alpha$, and then vary $\alpha$, analysing the wall-crossing in the spirit of \cite{thaddeus:1994,bradlow-garcia-prada-gothen:2004 triples}. It turns out that a crucial step in that proof (namely in the description of of $\N_{\alpha_m^-}(d)$) is a detailed analysis of the Hitchin fibration for $L$-twisted $\SL(2,\C)$-Higgs bundles, taking into account \emph{all} the fibres of the Hitchin map and not only the generic ones. This was done in \cite{gothen-oliveira:2013}, and we briefly describe this analysis.

In the last section of the paper we briefly mention some other results concerning the spaces $\N_\alpha(d)$, obtained in \cite{oliveira:2015}, that lead to the description of some geometric and topological properties of these moduli spaces. In particular, these results imply that, under some conditions on $d$ and on the genus of $X$, a Torelli type theorem holds for $\N_\alpha(d)$.
\section{From Higgs bundles to quadric bundles}
\subsection{Definitions and examples}

Let $X$ be a compact Riemann surface of genus $g\geq 2$, with canonical line bundle $K=T^*X^{1,0}$, the holomorphic cotangent bundle. Let $G$ be a real semisimple, connected, Lie group. Fix a maximal compact subgroup $H\subseteq G$ with complexification $H^\C\subseteq G^\C$. If $\liehc\subseteq\liegc$ are the corresponding Lie algebras, then the Cartan decomposition
is $\liegc=\liehc\oplus\liemc$, where $\liemc$ is the vector space defined as the orthogonal complement of $\liehc$ with respect to the Killing form. Since $[\liemc,\liehc]\subset\liemc$, then $\liemc$ is a representation of $H^\C$ via the
\emph{isotropy representation} $H^\C\to\GL(\liemc)$
induced by the adjoint representation $\Ad:G^\C\to\GL(\liegc)$.  If $E_{H^\C}$ is a principal $H^{\C}$-bundle over $X$, denote by $E_{H^\C}(\liemc)=E_{H^\C}\times_{H^{\C}}\liemc$ the vector bundle associated to $E_{H^\C}$ via the isotropy representation.

\begin{definition}\label{definition of Higgs bundle}
  A \emph{$G$-Higgs bundle} over $X$ is a
  pair $(E_{H^\C},\varphi)$ where $E_{H^\C}$ is a principal
  holomorphic $H^\C$-bundle and $\varphi$ is a global
  holomorphic section of $E_{H^\C}(\liemc)\otimes K$, called the
  \emph{Higgs field}.
\end{definition}

In practice we usually replace the principal $H^\C$-bundle $E_{H^\C}$ by the corresponding vector bundle associated to some standard representation of $H^\C$ in some $\C^n$. Let us give two examples.

If $G=\SL(n,\C)$, then $H^\C=G$ gives rise to a rank $n$ vector bundle $V$ with trivial determinant and since $\liemc=\mathfrak{sl}(n,\C)$, the Higgs field $\varphi$ is a traceless $K$-twisted endomorphism of $V$. 
If we fix the determinant of $V$ to be any line bundle and impose the same traceless condition to $\varphi:V\to V\otimes K$, then we also call the pair $(V,\varphi)$ an $\SL(n,\C)$-Higgs bundle, although it is really a ``twisted'' $\SL(n,\C)$-Higgs bundle. All these are usually just called \emph{Higgs bundles with fixed determinant}. These are the ``original'' Higgs bundles, introduced in \cite{hitchin:1987}.

If $G=\Sp(2n,\R)$, we can take $H=\U(n)$ as a maximal compact subgroup. So $H^\C=\GL(n,\C)$ gives rise to a rank $n$ holomorphic vector bundle $V$.
The Cartan decomposition is $\mathfrak{sp}(2n,\C)=\mathfrak{gl}(n,\C)\oplus\liemc$ where
the inclusion $\mathfrak{gl}(n,\C)\hookrightarrow\mathfrak{sp}(2n,\C)$ is given by
$A\mapsto\mathrm{diag}(A,-A^T)$. So $\liemc=\left\{(B,C)\in\mathfrak{gl}(n,\C)^2\st B=B^T,\ C=C^T\right\}$. Hence we have that:

\begin{definition}
An $\Sp(2n,\R)$-Higgs bundle is a triple $(V,\beta,\gamma)$ where $V$ is a holomorphic rank $n$ vector bundle, $\beta\in H^0(X,S^2V\otimes K)$ and $\gamma\in H^0(X,S^2V^*\otimes K)$.
\end{definition}

In an $\Sp(2n,\R)$-Higgs bundle $(V,\beta,\gamma)$, we can then think of $\gamma$ as a map $\gamma:V\to V^*\otimes K$ such that $\gamma^t\otimes \mathrm{Id}_K=\gamma$ and likewise for $\beta:V^*\to V\otimes K$.

\vspace{.5cm}

A $G$-Higgs bundle $(E_{H^\C},\varphi)$ is topologically classified by
the topological invariant of the corresponding $H^\C$-bundle
$E_{H^\C}$, given by an element $\pi_1(H)\cong\pi_1(G)$.

In \cite{garcia-prada-gothen-mundet:2008}, a general notion of
(semi,poly)stability of $G$-Higgs bundles was developed, allowing for
proving a Hitchin--Kobayashi correspondence between polystable
$G$-Higgs bundles and solutions to certain gauge theoretic equations
known as \emph{Hitchin equations}. On the other hand, A. Schmitt introduced stability
conditions for more general objects, which also apply for the $G$-Higgs bundles context, and 
used these in his general construction of moduli spaces; cf. \cite{schmitt:2008}. In
particular his stability conditions coincide with the ones relevant for the Hitchin--Kobayashi
correspondence.  It should be noted that the stability conditions
depend on a parameter $\alpha \in \sqrt{-1}\lieh \cap \liez$, where
$\liez$ is the centre of $\liehc$. In most cases this parameter is fixed by the topological type, so it really does not play any relevant role. This happens for any compact or complex Lie group and most real groups. Indeed, the only case where the parameter is not fixed by the topology is when $G$ is of hermitian type. This is the case of $\Sp(2n,\R)$, so it is important for us to take into account the presence of $\alpha$. 

Denote by $\mathcal{M}^\alpha_d(G)$ the moduli space of $S$-equivalence classes of $\alpha$-semistable $G$-Higgs bundles with topological invariant $d \in \pi_1(G)$. On each $S$-equivalence class there is a unique (up to isomorphism) $\alpha$-polystable representative, so we can consider $\mathcal{M}^\alpha_d(G)$ as the moduli space isomorphism classes of $\alpha$-polystable $G$-Higgs bundles. 

\begin{remark}\label{rmk:Ltwisting}
Given any line bundle $L\to X$, of non-negative degree, everything we just said generalises to \emph{$L$-twisted $G$-Higgs pairs}. The only difference to $G$-Higgs bundles is that the Higgs field is a section of $E_{H^\C}(\liemc)\otimes L$ instead of $E_{H^\C}(\liemc)\otimes K$.
\end{remark}

\subsection{Higgs bundles for $\Sp(4,\R)$ and quadric bundles}

\subsubsection{Moduli of $\Sp(4,\R)$-Higgs bundles}

We already know that an $\Sp(4,\R)$-Higgs bundle is a triple $(V,\beta,\gamma)$ with $\rk(V)=2$ and \[\beta\in H^0(X,S^2V\otimes K),\quad \gamma\in H^0(X,S^2V^*\otimes K).\] The topological type is given by the degree of $V$: $d=\deg(V)\in\Z=\pi_1(\Sp(4,\R))$.
In fact, $\Sp(4,\R)$ is of hermitian type, and the invariant $d$ is the Toledo invariant mentioned in Section \ref{sec:1}.

Given a real parameter $\alpha$, here is the $\alpha$-(semi)stability condition for $\Sp(4,\R)$-Higgs bundles; see \cite{garcia-prada-gothen-mundet:2008,garcia-prada-gothen-mundet:2013} for the deduction of these conditions.

\begin{definition}\label{def:semistabSp(4,R)}
Let $(V,\beta,\gamma)$ be an $\Sp(4,\R)$-Higgs bundle with $\deg(V)=d$. It is $\alpha$-semistable if the following hold:
\begin{enumerate}
\item if $\beta=0$ then $d-2\alpha\geq 0$;
\item if $\gamma=0$ then $d-2\alpha\leq 0$.
\item for any line subbundle $L\subset V$, we have:
\begin{enumerate}
\item $\deg(L)\leq\alpha$ if $\gamma(L)=0$;
\item $\deg(L)\leq d/2$ if $\beta(L^\perp)\subset L\otimes K$ and $\gamma(L)\subset L^\perp\otimes K$;
\item $\deg(L)\leq d-\alpha$ if $\beta(L^\perp)=0$.
\end{enumerate}
\end{enumerate}
\end{definition}

Here $L^\perp$ stands for the kernel of the projection $V^*\to L^*$, so it is the annihilator of $L$ under $\gamma$; note that we are not considering any metric on $V$ whatsoever. 
As usual, there are also the notions of $\alpha$-stability (by considering strict inequalities) and of $\alpha$-polystability; cf. \cite{gothen-oliveira:2012}. 

\begin{remark}\label{rmk:stab of underlying bundle}
If we view a semistable rank two vector bundle $V$ of degree $d$ as an $\Sp(4,\R)$-Higgs bundle with $\beta=\gamma=0$, then it is $\alpha$-semistable if and only if $\alpha=d/2$.
\end{remark}

Our aim is to present an overview on the study of the connected components of the moduli space of $0$-polystable $\Sp(4,\R)$-Higgs bundles $\M_d(\Sp(4,\R))$ for certain values of $d$. To keep the notation simpler, we will just write
\[\M_d(\Sp(4,\R))=\M^0_d(\Sp(4,\R))\] for the case $\alpha=0$. In this case we will just say ``polystable'' instead of $0$-polystable and likewise for stable and semistable.

\begin{remark}\textbf{(Relation with representations $\pi_1(X)\to\Sp(4,\R)$)}\label{rmk:NonabelianHodge}
We consider $\alpha=0$ because this is the appropriate value for which non-abelian Hodge theory applies. More precisely, the \emph{non-abelian Hodge Theorem} for $\Sp(4,\R)$ states that an $\Sp(4,\R)$-Higgs bundle is polystable if and only if it corresponds to a reductive representation of $\pi_1(X)$ in $\Sp(4,\R)$. This implies that $\M_d(\Sp(4,\R))$ is homeomorphic to the space of reductive representations of $\pi_1(X)$ in $\Sp(4,\R)$, with topological invariant $d$, modulo the action of conjugation by $\Sp(4,\R)$, that is to $\cR_d(\Sp(4,\R))=\mathrm{Hom}^{\mathrm{red}}(\pi_1(X),\Sp(4,\R))/\Sp(4,\R)$. This theorem is in fact valid for any real semisimple Lie group and also for real reductive groups with some slight modifications. The proof in the classical $G=\SL(n,\C)$ case follows from \cite{corlette:1988,donaldson:1987,hitchin:1987,simpson:1992}. The more general case follows from \cite{corlette:1988,garcia-prada-gothen-mundet:2008}. See for instance \cite{garcia-prada-mundet:2004,garcia-prada-gothen-mundet:2013} for more information for the case of $\Sp(2n,\R)$ and \cite{bradlow-garcia-prada-gothen:2005} for an overview on the approach for the general group case.
\end{remark}

The \emph{Milnor-Wood inequality} for $G=\Sp(4,\R)$ states that if an $\Sp(4,\R)$-Higgs bundle of degree $d$ is semistable, then \cite{domic-toledo:1987,gothen:2001,garcia-prada-gothen-mundet:2008,biquard-garcia-prada-rubio:2015}\[|d|\leq 2g-2.\] (A similar type of inequality was proved for the first time for $G=\PSL(2,\R)$ by Milnor in \cite{milnor:1958}, on the representations side; cf. Remark \ref{rmk:NonabelianHodge}.) 

So $\M_d(\Sp(4,\R))$ is empty if $|d|>2g-2$. 
If $|d|=2g-2$ then we say that we are in the \emph{maximal Toledo} case, which is in fact the case where more interesting phenomena occur. Indeed, it is known \cite{gothen:2001} that $\M_{\pm(2g-2)}(\Sp(4,\R))$ has $3\times 2^{2g}+2g-4$ components and that it is isomorphic to the moduli space of $K^2$-twisted $\GL(2,\R)$-Higgs bundles --- this is an example of the Cayley partner phenomena mentioned in the introduction (see also \cite{bradlow-garcia-prada-gothen:2005,biquard-garcia-prada-rubio:2015}). In subsection \ref{subsect:maxToledo} below we will see this for a subvariety of $\M_{2g-2}(\Sp(4,\R))$. It is also known that $\M_0(\Sp(4,\R))$ is connected \cite{gothen:2001}. The corresponding results for these two extreme cases for $|d|$ in higher rank are also known; cf. \cite{garcia-prada-gothen-mundet:2013}.

Nevertheless, in this paper we are interested in the components of $\M_d(\Sp(4,\R))$ for \emph{non-maximal} and \emph{non-zero} Toledo invariant: $0<|d|<2g-2$.
The duality $(V,\beta,\gamma)\mapsto(V^*,\gamma,\beta)$ gives an isomorphism $\M_d(\Sp(4,\R))\cong\M_{-d}(\Sp(4,\R))$, thus we just consider $0<d<2g-2$.

\subsubsection{The approach to count components}

The general idea, introduced in \cite{hitchin:1987,hitchin:1992}, to study the connected components of $\M_c(G)$ is to use the functional $f:\M_c(G)\to\R$ mapping a $G$-Higgs bundle to the (square of the) $L^2$-norm of the Higgs field. The fact that $f$ is proper and bounded below implies that it attains a minimum on each connected component of $\M_c(G)$. Hence the number of connected components of $\M_c(G)$ is bounded above by the one of the subvariety $\N_c(G)\subset\M_c(G)$ of local minimum  of $f$. The procedure is thus to identify $\N_c(G)$, study its connected components and then draw conclusions about the components of $\M_c(G)$. Of course if $\N_c(G)$ turns out to be connected, then it immediately follows that $\M_c(G)$ is connected as well.

Explicitly, for $\Sp(4,\R)$, the Higgs field splits as $\beta$ and $\gamma$, so we have 
\begin{equation}\label{eq:Hitchin functional}
f(V,\beta,\gamma)=\|\beta\|_{L^2}^2+\|\gamma\|_{L^2}^2=\int_X\tr(\beta\beta^{*,h})+\int_X\tr(\gamma^{*,h}\gamma),
\end{equation} where $h:V\to\bar{V}^*$  is the metric on $V$ which provides the Hitchin-Kobayashi correspondence and hence we are taking in \eqref{eq:Hitchin functional} the adjoint with respect to $h$.

The following result completely identifies the subvariety of local minima in the non-zero and non-maximal cases. For this identification it is important that, over the smooth locus of $\M_d(\Sp(4,\R))$, the function $f$ is a moment map of the hamiltonian circle action $(V,\beta,\gamma)\mapsto(V,e^{i\theta}\beta,e^{i\theta}\gamma)$. By work of Frankel \cite{frankel:1959}, a smooth point of $\M_d(\Sp(4,\R))$ is a
critical point of $f$ exactly when it is a fixed point of this $\U(1)$-action. Then there is a cohomological criteria \cite[Corollary 4.15]{bradlow-garcia-prada-gothen:2003} which identifies the local minima among this fixed point set. Finally one has to perform a subsequent analysis to identify the local minima over the singular locus of $\M_d(\Sp(4,\R))$.

\begin{proposition}[\cite{gothen:2001}]\label{prop:localmin}
Let $(V,\beta,\gamma)$ represent a point of $\M_d(\Sp(4,\R))$, with $0<d<2g-2$. Then it is a local minimum of $f$ if and only if $\beta=0$.
\end{proposition}

Thus, for $0<d< 2g-2$, the subvariety of local minima $\N_d(\Sp(4,\R))\subset\M_d(\Sp(4,\R))$ is given by pairs $(V,\gamma)$ where $V$ is a rank $2$ bundle, of degree $d$ and $\gamma$ is a section of $S^2V^*\otimes K$. This is what we call a \emph{quadric bundle}. Since $d$ is positive, $\gamma$ must indeed be a non-zero section, as we saw in Remark \ref{rmk:stab of underlying bundle}. 

\begin{definition}\label{definition of U quadric bundle}
  A \emph{quadric bundle} on $X$ is a pair $(V,\gamma)$, where $V$
  is a holomorphic vector bundle over $X$ and $\gamma$ is a holomorphic non-zero section of $S^2V^*\otimes K$.
\end{definition}

Quadric bundles are sometimes also called \emph{conic bundles} or \emph{quadratic pairs} in the literature. In particular, this happens in the papers \cite{gothen-oliveira:2012,oliveira:2015} by the author where they were named quadratic pairs. But the term ``quadric bundles'' used in \cite{giudice-pustetto:2014} is indeed more adequate, since it is more specific and moreover reveals the fact that these can be seen as bundles of quadrics, since for each $p\in X$ the map $\gamma$ restricted to the fibre $V_p$ defines a bilinear symmetric form, hence a quadric in $\mathbb{P}^{\rk(V)-1}$. When $\rk(V)=2$, the term conic bundle is then perfectly adequate also.

The rank and degree of a quadric bundle are of course the rank and degree of $V$. We will only consider the rank $2$ case. The rank $n$ case appears naturally by considering $\Sp(2n,\R)$-Higgs bundles.

\begin{remark}\label{rmk:U-quad}
More generally, one can define $U$-quadric bundles, for a fixed holomorphic line bundle $U$ over $X$. The only difference for the preceding definition is that $\gamma$ is a non-zero section of $S^2V^*\otimes U$. We will mostly be interested in ($K$-)quadric bundles, but more general $U$-quadric bundles will also appear, more precisely when $U=LK$, for some line bundle $L$, in relation with the group $\SO_0(2,3)$.
All results below can be adapted to this more general setting \cite{gothen-oliveira:2012}.
\end{remark}

Quadric bundles of rank up to $3$ were studied in \cite{gomez-sols:2000} by G\'omez and Sols, where they introduced an
appropriate $\alpha$-semistability condition, depending on a real parameter $\alpha$, and
constructed moduli spaces of $S$-equivalence classes of $\alpha$-semistable quadric bundles using GIT. The construction of the
moduli spaces follows from the general methods of \cite{schmitt:2008}. Denote the moduli space
of $S$-equivalence classes of $\alpha$-semistable $U$-quadric bundles
on $X$ of rank $2$ and degree $d$ by $\N_{X,\alpha}(2,d)=\N_\alpha(d)$.


A simplified $\delta$-(semi)stability condition for quadric bundles of arbitrary rank has been obtained in \cite{giudice-pustetto:2014}. 
In rank $2$ our $\alpha$-semistability condition reads as follows (see \cite[Proposition 2.15]{gothen-oliveira:2012}). It is equivalent to the corresponding one on \cite{giudice-pustetto:2014} by taking $\alpha=d/2-\delta$.
\begin{definition}\label{simplified notion of polystability}
Let $(V,\gamma)$ be a rank 2 quadric bundle of degree $d$.
\begin{itemize}
\item The pair $(V,\gamma)$ is \emph{$\alpha$-semistable} if and only if
  $\alpha\leq d/2$ and, for any line bundle $L\subset V$, the
  following conditions hold:
\begin{enumerate}
\item $\deg(L)\leq\alpha$, if $\gamma(L)=0$;
\item $\deg(L)\leq d/2$, if $\gamma(L)\subset L^\perp K$;
\item $\deg(L)\leq d-\alpha$, if $\gamma(L)\not\subset L^\perp K$.
\end{enumerate}
\item The pair $(V,\gamma)$ is \emph{$\alpha$-stable} if and only if it is
  $\alpha$-semistable for any line bundle $L\subset V$, the conditions
  $(1)$, $(2)$ and $(3)$ above hold with strict inequalities.
\end{itemize}
\end{definition}

Clearly these conditions are compatible with the ones of Definition \ref{def:semistabSp(4,R)}. There is also the notion of $\alpha$-polystability, but we omit it (see again Proposition 2.15 of \cite{gothen-oliveira:2012}). The important thing to note is that on each $S$-equivalence class of $\alpha$-semistable quadric bundles there is a unique $\alpha$-polystable representative. Thus the points of $\N_\alpha(d)$ parametrize the isomorphism classes of $\alpha$-polystable quadric bundles of rank $2$ and, furthermore, $\N_\alpha(d)$ is a subvariety of $\M_d^\alpha(\Sp(4,\R))$.

The next result follows from Proposition \ref{prop:localmin} and the discussion preceding it.

\begin{proposition}\label{prop:boundcomp}
Let $0<d<2g-2$. The number of connected components of the moduli space $\M_d(\Sp(4,\R))$ of semistable $\Sp(4,\R)$-Higgs bundles of degree $d$ is bounded above by the number of connected components of $\N_0(d)$, the moduli space of $0$-polystable quadric bundles of degree $d$.
\end{proposition}

\section{Moduli of quadric bundles and wall-crossing}

\subsection{Non-emptiness conditions}

The next result gives a Milnor-Wood type of inequality for quadric bundles. 
\begin{proposition}\label{empty above d_U}
 If $\N_\alpha(d)$ is non-empty then $2\alpha\leq d\leq 2g-2$.
\end{proposition}
\proof The first statement is immediate from $\alpha$-semistability, hence let us look to the second inequality.

Let $(V,\gamma)$ be quadric bundle of rank $2$ and degree $d$. If $\rk(\gamma)=2$ (generically), then $\det(\gamma)$ is a
non-zero section of $\Lambda^2V^{-2}K^2$ so $d\leq 2g-2$.

Suppose now that there exists an $\alpha$-semistable quadric bundle $(V,\gamma)$ of rank $2$ and degree
$d>2g-2$, with $\rk(\gamma)<2$. Since $\gamma\neq 0$, we must
have $\rk(\gamma)=1$.  Let $N$ be the line subbundle of $V$ given by the kernel of $\gamma$ and let
$I\subset V^*$ be such that $IK$ is the saturation of the image sheaf of $\gamma$. Hence $\gamma$ induces a non-zero map of line bundles
$V/N\to IK$, so
\begin{equation}\label{non-zero map}
 -d+\deg(N)+\deg(I)+2g-2\geq 0.
\end{equation} But, from the $\alpha$-semistability condition, we have $\deg(N)\leq\alpha$ and $\deg(I)\leq\alpha-d$, because $\gamma(I^\perp)=0$. This implies
$-d+\deg(N)+\deg(I)+2g-2<0$, contradicting \eqref{non-zero map}. We conclude that there is no such $(V,\gamma)$.
\endproof

In fact, the inequalities of this proposition are equivalent to the non-emptiness of the moduli. This follows from the results below.
So from now on we assume \[2\alpha\leq d\leq 2g-2.\] Indeed most of the times we will consider $2\alpha< d< 2g-2$.

\begin{remark}
Recall that our motivation comes from $\Sp(4,\R)$-Higgs bundles and there (see Proposition \ref{prop:localmin}) we imposed $d>0$. However, the moduli spaces of quadric bundles make perfect sense and can be non-empty also for $d\leq 0$. Hence we will \emph{not} impose $d>0$ for the quadric bundles moduli spaces, although when $d\leq 0$ we lose the relation with Higgs bundles.
\end{remark}

\subsection{Moduli for small parameter}

\subsubsection{Stabilization parameter}

For a fixed $d\leq 2g-2$, we know that there are no moduli spaces $\N_\alpha(d)$ whenever $\alpha$ is ``large'', meaning $\alpha >d/2$. Here we prove that there is a different kind of phenomena when $\alpha$ is ``small''. Precisely, we show that all the moduli spaces $\N_{\alpha'}(d)$ are isomorphic for any $\alpha'<d-g+1$. Moreover, in all of them, the map $\gamma$ is generically non-degenerate.
Write $\alpha_m=d-g+1$.

\begin{proposition}\label{injectivity-stabilization}
  If $(V,\gamma)$ is an $\alpha$-semistable pair with $\alpha<\alpha_m$, then $\gamma$ is generically non-degenerate. Moreover, if $\alpha_2\leq\alpha_1<\alpha_m$, then 
  $\N_{\alpha_1}(d)$ and $\N_{\alpha_2}(d)$ are isomorphic.
\end{proposition}
\begin{proof} 
Recall that we always have $\gamma\neq 0$. If $\rk(\gamma)=1$,
considering again the line bundles $N=\ker(\gamma)\subset V$ and
$I\subset V^*$ as in the proof of Proposition \ref{empty above d_U}, we see that
$0\leq -d+\deg(N)+\deg(I)+2g-2\leq 2\alpha-2d+2g-2$, 
i.e., $\alpha\geq\alpha_m$. This settles the first part of the proposition.

Let $(V,\gamma)\in\N_{\alpha_1}(d)$. The only way that $(V,\gamma)$ may not belong to $\N_{\alpha_2}(d)$ is from the existence of an $\alpha_2$-destabilizing subbundle which, since $\alpha_2\leq\alpha_1$ and looking at Definition \ref{simplified notion of polystability}, must be a line subbundle $L\subset V$ such that $\gamma(L)=0$ and $\deg(L)>\alpha_2$. This in turn implies that $\rk(\gamma)=1$ generically, which is impossible due to the first part of the proposition.

Conversely, if $(V,\gamma)\in\N_{\alpha_2}(d)$, then $(V,\gamma)\in\N_{\alpha_1}(d)$ unless there is an $\alpha_1$-destabilizing subbundle $L$ of $(V,\gamma)$ such that $d-\alpha_1<\deg(L)\leq d-\alpha_2$, and $\gamma(L)\not\subset L^\perp K$. Therefore the composite 
$L\to V\xrightarrow{\gamma} V^*\otimes U\to L^{-1}K$ is non-zero so $-2\deg(L)+2g-2\geq 0$. On the other hand, $d-\alpha_1<\deg(L)$ together with $\alpha_1<\alpha_m$, gives
$-2\deg(L)+2g-2<0$, yielding again to a contradiction.
\end{proof}

We now aim to study the connectedness of the spaces $\N_\alpha(d)$, for $\alpha<\alpha_m$. Although our main motivation comes from the study of $\Sp(4,\R)$-Higgs bundles with non-maximal Toledo invariant (cf. Proposition \ref{prop:localmin}), let us say a few words about $\N_\alpha(2g-2)$, which really has a different behaviour from all the other cases. 

\subsubsection{Maximal Toledo invariant}\label{subsect:maxToledo}

Take $d=2g-2$. In this case $\alpha_m=g-1=d/2$, so the stabilisation parameter of the previous results is really the largest value for which non-emptiness holds. This means that, whenever non-empty, \emph{all the moduli spaces $\N_\alpha(2g-2)$ are isomorphic}, independently of $\alpha$. Accordingly, in this maximal case, we drop the $\alpha$ from the notation and just write $\N(2g-2)$.

The other special feature about this case is that if $(V,\gamma)\in\N(2g-2)$, then $\gamma:V\to V^*\otimes K$ is an isomorphism, since we already know that it must be injective and now the degrees match. By choosing a square
  root $K^{1/2}$ of $K$, $\gamma$ gives rise to a symmetric
  isomorphism $q=\gamma\otimes \mathrm{Id}_{K^{-1/2}}:V\otimes K^{-1/2}\cong V^*\otimes K^{1/2}$, i.e. to a
  nowhere degenerate quadratic form on $V\otimes
  K^{-1/2}$. In other words, $(V\otimes K^{-1/2},q)$ is an orthogonal vector bundle. 
Now, there is a semistability condition for orthogonal bundles (namely that any isotropic subbundle must have non-positive degree; \cite{ramanathan:1975}), and it can be seen that the orthogonal bundle $(V\otimes  K^{-1/2},q)$ is semistable if and only if
  $(V,\gamma)$ is $\alpha$-semistable for any $\alpha<\alpha_m$. So:
  
  \begin{proposition}
The moduli space $\N(2g-2)$ is isomorphic to the moduli space of rank $2$ orthogonal vector bundles (without fixed topological type).
  \end{proposition}

The existence of this isomorphism justifies the disconnectedness of $\N(2g-2)$. This is an example of the Cayley correspondence mentioned in the introduction. 
All this goes through higher rank, telling us that quadric bundles are the natural generalisation of orthogonal vector bundles, when we remove the non-degeneracy condition, providing another motivation for the consideration of these objects.

\subsubsection{Quadric bundles, twisted Higgs pairs and the fibres of the Hitchin map}\label{The Hitchin map}

Write $\alpha_m^-$ for any value of $\alpha$ less than $\alpha_m=d-g+1$. We shall now deal with the spaces $\N_{\alpha_m^-}(d)$ for any $d<2g-2$. We will do it by relating pairs $(V,\gamma)$ with certain twisted rank $2$ Higgs bundles and using the Hitchin map on the corresponding moduli space.

Consider a quadric bundle $(V,\gamma)\in\N_{\alpha_m^-}(d)$. By Proposition \ref{injectivity-stabilization}, $\det(\gamma)$ is a non-zero holomorphic section of $\Lambda^2V^{-2}K^2$. Since now $d<2g-2$, the section $\det(\gamma)$ has zeros, so we consider the corresponding effective divisor $\divisor(\det(\gamma))\in \cSym^{4g-4-2d}(X)$.

Write $\Jac^d(X)$ for the ``Jacobian variety'' of degree $d$ holomorphic line bundles over $X$. Let $\mathcal P_X$ be the $2^{2g}$-cover of $\cSym^{4g-4-2d}(X)$ obtained by pulling back the cover $\Jac^{2g-2-d}(X)\to \Jac^{4g-4-2d}(X)$, $L\mapsto L^2$, under the Abel-Jacobi map $\cSym^{4g-4-2d}(X)\to\Jac^{4g-4-2d}(X)$.
The elements of $\mathcal P_X$ are pairs $(D,L)$ in the product $\cSym^{4g-4-2d}(X)\times\Jac^{2g-2-d}(X)$ such that $\mathcal{O}(D)\cong L^2$.

In order to describe $\N_{\alpha_m^-}(d)$, we shall use the following map, which is analogue to the so-called Hitchin map defined by Hitchin in \cite{hitchin:1987}, and which will recall below in \eqref{eqHitchinmap}:
\begin{equation}\label{def of Hitchin map}
\begin{array}{rccl}
  h:&\N_{\alpha_m^-}(d) & \longrightarrow & \mathcal P_X\\
  &(V,\gamma) & \longmapsto & (\divisor(\det(\gamma)),\Lambda^2V^{-1}K).
\end{array}
\end{equation}
Our goal is to be able to say something about the fibres of this map.

To relate $h$ with the Hitchin map, recall first that, given any line bundle $L$ of non-negative degree, an \emph{$L$-twisted Higgs pair} of type is a pair $(V,\varphi)$, where $V$ is a holomorphic vector bundle over $X$ and $\varphi\in H^0(X,\End(V)\otimes L)$. So, we are just twisting the Higgs field by $L$ instead of $K$. 


\begin{definition}\label{def:sstabbundlerank2}
A rank $2$ and degree $d$, $L$-twisted Higgs pair $(V,\varphi)$ is \emph{semistable} if $\deg(F)\leq d/2$ for any line subbundle $F\subset V$ such that $\varphi(F)\subset FL$.
\end{definition}

Let $\M_{L}^{\Lambda}$ denote the moduli space of $L$-twisted Higgs pairs of rank $2$ and degree $d$, with fixed determinant $\Lambda\in\mathrm{Jac}^d(X)$ and with traceless Higgs field.
In this particular case, the \emph{Hitchin map} in $\M_L^\Lambda$ is defined by:
\begin{equation}\label{eqHitchinmap}
\begin{array}{rccl}
  \mathcal H:&\M_L^\Lambda & \longrightarrow & H^0(X,L^2)\\
  &(V,\varphi) & \longmapsto & \det(\varphi).
\end{array}
\end{equation}

We can naturally associate a $\xi$-twisted Higgs pair to a given quadric bundle $(V,\gamma)$, of rank $2$, where $\xi=\Lambda^2V^{-1}K$. This is done by taking advantage of the fact that for a $2$-dimensional vector space $\VV$, there is an isomorphism $\VV\otimes \Lambda^2\VV^*\cong\VV^*$ given by $v\otimes\phi\mapsto\phi(v\wedge -)$.
Then, from such quadric bundle, simply associate the $\xi$-twisted Higgs pair $(V,g^{-1}\gamma)$, where $g$ is the isomorphism
\begin{equation}\label{isomorphism Vdual VtensordetV}
g:V\otimes\xi\stackrel{\cong}{\longrightarrow} V^*\otimes K
\end{equation}
given by $g(v\otimes\phi\otimes u)=\phi(v\wedge -)\otimes u$; so indeed $g^{-1}\gamma:V\to V\otimes\xi$.
Choosing appropriate local frames, $g$ is locally given by $\left(\begin{smallmatrix}0 & -1 \\1 & 0\end{smallmatrix}\right)$ so 
\begin{equation}\label{eq:char-twistedHiggs}
\det(g^{-1}\gamma)=\det(\gamma)\,\text{ and }\,\tr(g^{-1}\gamma)=0,
\end{equation} due to the symmetry of $\gamma$.
Moreover, it is easy to see that $(V,\gamma)$ is $\alpha_m^-$-semistable if and only if the corresponding $(V,g^{-1}\gamma)$ is semistable as in Definition \ref{def:sstabbundlerank2}. So if $(V,\gamma)\in\N_{\alpha_m^-}$, then $(V,g^{-1}\gamma)$ represents a point in $\M^{\xi^{-1}K}_\xi$.

Let us now go back to the map $h$ in \eqref{def of Hitchin map}. Let $(D,\xi)$ be any pair in $\mathcal P_X$. We want to describe the fibre of $h$ over $(D,\xi)$, i.e., the space of isomorphism classes of $\alpha_m^-$-polystable quadric bundles $(V,\gamma)$ such that $\divisor(\det(\gamma))=D$ and $\Lambda^2V\cong \xi^{-1}K$.
The following result gives the fibre $h^{-1}(D,\xi)$ in terms of the fibre $\mathcal H^{-1}(s)$, for a certain section $s$ of $\xi^2$.
\begin{proposition}[\cite{gothen-oliveira:2012}]\label{fibres quad and Higgs isomorphic}
Let $(D,\xi)\in \mathcal P_X$ and choose some $s\in H^0(X,\xi^2)$ such that $\divisor(s)=D$. Then $h^{-1}(D,\xi)\in\N_{\alpha_m^-}(d)$ is isomorphic to $\mathcal H^{-1}(s)\in\M_{\xi}^{\xi^{-1}K}$.
\end{proposition}

The isomorphism of this proposition is of course given by the above correspondence between quadric bundles and $\xi$-twisted Higgs pairs. Notice that everything makes sense because of \eqref{eq:char-twistedHiggs}.
A word of caution is however required here since there is a choice of a section $s$ associated to the divisor $D$ in Proposition \ref{fibres quad and Higgs isomorphic}. 
However, the given description of $h^{-1}(D,\xi)$ does not depend of this choice, due to Lemma 4.6 of \cite{gothen-oliveira:2012}; see also Remark 4.10 in loc. cit. for more details. 

Using this we can prove the following.

\begin{theorem}[\cite{gothen-oliveira:2012}]\label{thm:Nalpham}
Let $d<2g-2$. The moduli space $\N_{\alpha_m^-}(d)$ is connected and has dimension $7g-7-3d$.
\end{theorem}

The basic idea to prove connectedness is to prove that any fibre of $h$ is connected. For that we use Proposition \ref{fibres quad and Higgs isomorphic} and want to prove that $\mathcal H^{-1}(s)$ is connected for \emph{every} $0\neq s\in H^0(X,\xi^2)$. This is done using the theory of spectral covers and their Jacobians and Prym varieties, as developed in \cite{beauville-narasimhan-ramanan:1989,hitchin:1987,hitchin:1987b}. Besides these classical references, the reader may also check the details of the following definitions for instance in \cite{gothen-oliveira:2013}.

For every $s\neq 0$, there is a naturally associated curve $X_s$ --- the \emph{spectral curve} of $s$ --- inside the total space of $\pi:\xi\to X$. The projection $\pi|_{X_s}:X_s\to X$ is a $2:1$ cover of $X$, with the branch locus being given by the divisor of $s$.

For generic $s\in H^0(X,\xi^2)$ the curve $X_s$ is smooth. It is well-known that $\mathcal H^{-1}(s)$ is indeed (a torsor for) the \emph{Prym variety} of $X_s$. This Prym variety is, in particular, a complex torus, so connected. 
If $\deg(\xi)\geq 2g-2$ then the connectedness of every fibre of $\mathcal H$ follows from the connectedness of the generic fibre (see Proposition 3.7 of \cite{gothen-oliveira:2013}).
It is nevertheless important to notice that $\deg(\xi)=-d+2g-2$, so $\deg(\xi)$ can be any positive integer. Moreover, it is precisely the case $\deg(\xi)<2g-2$ that is of most interest to us, since that is the case relevant to $\Sp(4,\R)$-Higgs bundles. So, for these cases, our knowledge of the generic fibre is not enough to draw conclusions on the connectedness of the singular fibres, that is, the ones where the spectral curve $X_s$ acquires singularities. However, this was achieved by P. Gothen and the author in \cite{gothen-oliveira:2013} as follows. 

When the spectral case is irreducible, we use the correspondence between Higgs pairs on $X$ and rank one torsion free sheaves on $X_s$ \cite{beauville-narasimhan-ramanan:1989} to show that the fibre of the Hitchin map is essentially the compactification by rank $1$ torsion free sheaves of the Prym of the
double cover $X_s \to X$. In order to prove the connectedness of the fibre, we made use of
the the compactification of the Jacobian of $X_s$ by the \emph{parabolic
modules} of Cook \cite{cook:1993,cook:1998}. One advantage of this compactification is that it
fibres over the Jacobian of the normalisation of $X_s$, as opposed to
the compactification by rank one torsion free sheaves.
In the case of reducible spectral curve $X_s$, we gave a direct description of the fibre as a stratified space.
All together, the statement of our result, adapted to the situation under consideration in Proposition \ref{fibres quad and Higgs isomorphic}, is the following.

\begin{theorem}[\cite{gothen-oliveira:2013}]
Consider the Hitchin map $\mathcal H:\M_{\xi}^{\xi^{-1}K}\to H^0(X,\xi^2)$. For any $s$, $\mathcal H^{-1}(s)$ is connected. Moreover, for $s\neq 0$, the dimension of the fibre is $\dim(\mathcal H^{-1}(s))=-d+3g-3$.
\end{theorem}

As $\mathcal P_X$ is connected and of dimension $4g-4-2d$, this settles Theorem \ref{thm:Nalpham}.

The following corollary of Theorem \ref{thm:Nalpham} is immediate.

\begin{corollary}\label{cor:connectednessN-;g-1<d<2g-2}
If $g-1<d<2g-2$, then the moduli space $\N_0(d)$ is connected.
\end{corollary}

For the cases $0<d<g-1$, we must take into account other values of the parameter and not just $\alpha_m^-$.

\subsection{Critical values and wall-crossing}

Having established the connectedness of the space $\N_{\alpha_m^-}(d)$ in Theorem \ref{thm:Nalpham}, the purpose of this section is to study the variation of the moduli
spaces $\mathcal{N}_{\alpha}(d)$ with the stability parameter $\alpha$.  Recall that the goal is to be able to say something about the connectedness of $\N_0(d)$.
As in several other cases \cite{thaddeus:1994,bradlow-garcia-prada-gothen:2004 triples},
 we have \emph{critical values} for the parameter. These are special values $\alpha_k$, for which the $\alpha$-semistability condition changes.
One proves that indeed there are a finite number of these critical values and, more precisely, that $\alpha$ is a critical value if and only if it is equal to $d/2$ or to $\alpha_k=[d/2]-k$, with $k=0,\ldots,-d+g-1+[d/2]$.
By definition, on each open interval between consecutive critical values, the $\alpha$-semistability condition does not vary, hence the corresponding moduli spaces are isomorphic.
If $\alpha_k^+$ denotes the value of any parameter between the critical values $\alpha_k$ and $\alpha_{k+1}$, we can write without ambiguity 
$\N_{\alpha_k^+}(d)$
for the moduli space of $\alpha_k^+$-semistable quadric bundles of
 for any $\alpha$ between $\alpha_k$ and $\alpha_{k+1}$. Likewise, define $\N_{\alpha_k^-}(d)$, with $\alpha_k^-$ denoting any value between the critical values $\alpha_{k-1}$ and $\alpha_k$. With this notation we have $\N_{\alpha_k^+}(d)=\N_{\alpha_{k+1}^-}(d)$.

The information obtained so far on the variation of $\mathcal{N}_{\alpha}(d)$
with $\alpha$ and $d$ is summarised in the next graphic. 

\vspace{.5cm}

\noindent\begin{tabular}{l l}
\begin{minipage}{0.6\textwidth}
\includegraphics[scale=1.3]{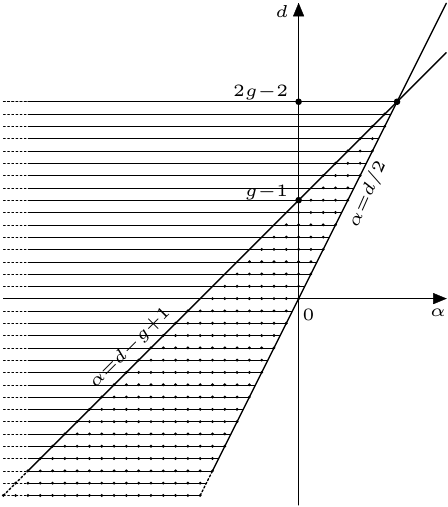}
\end{minipage}
\begin{minipage}{0.375\textwidth}
Above the line
  $d=2g-2$, $\N_\alpha(d)$ is empty as well as on the right of the
  line $\alpha=d/2$. For a fixed $d<2g-2$, the region on the left of the line
  $\alpha=\alpha_m=d-g+1$, is the region $\N_{\alpha_m^-}(d)$ described in the previous section, where there are no critical
  values. The critical values are represented by the dots between the lines $\alpha=\alpha_m$ and $\alpha=d/2$.
\end{minipage}
&
\end{tabular}

\vspace{.5cm}

Given a critical value $\alpha_k$ we have the corresponding subvariety $\cS_{\alpha_k^+}(d)\subset\N_{\alpha_k^+}(d)$ consisting of those pairs which are
$\alpha_k^+$-semistable but $\alpha_k^-$-unstable. In the same manner, define the subvariety $\cS_{\alpha_k^-}(d)\subset\N_{\alpha_k^-}(d)$.
Consequently,
\begin{equation}\label{N--S-=N+-S+}
  \N_{\alpha_k^-}(d)\setminus\cS_{\alpha_k^-}(d)
  \cong\N_{\alpha_k^+}(d)\setminus\cS_{\alpha_k^+}(d).
\end{equation} So the spaces $\cS_{\alpha_k^\pm}(d)$ encode the difference between the spaces $\N_{\alpha_k^-}(d)$ and $\N_{\alpha_k^+}(d)$ on opposite sides of the critical value $\alpha_k$. This difference is usually known as the \emph{wall-crossing} phenomena through $\alpha_k$. In order to study this wall-crossing we need a description of the spaces $\cS_{\alpha_k^\pm}(d)$. 

In Section 3 of \cite{gothen-oliveira:2012} we studied these $\cS_{\alpha_k^\pm}(d)$ for any critical value.
In particular, it was enough to conclude that they have high codimension in $\N_{\alpha_k^\pm}(d)$. For technical reasons, we had to impose the condition $d<g-1$. More precisely, we have the following.

\begin{proposition}\label{prop:codimS+-}
Suppose that $d<g-1$. Then $\dim \N_\alpha(d)=7g-7-3d$, for any $\alpha\leq d/2$. Moreover, for any $k$, the codimensions of $\cS_{\alpha_k^\pm}(d)\subset\N_{\alpha_k^\pm}(d)$ are strictly positive. 
\end{proposition}

\subsection{Connectedness of the moduli spaces of quadric bundles}

Since we already know from Theorem \ref{thm:Nalpham} that $\N_{\alpha_m^-}(d)$ is connected for any $d<2g-2$, Proposition \ref{prop:codimS+-} yields the following (cf. \cite[Theorem 5.3]{gothen-oliveira:2012}):

\begin{theorem}
The moduli spaces $\N_\alpha(d)$ are connected for any $d<g-1$ and any $\alpha<d/2$.
\end{theorem}

\begin{corollary}\label{cor:connectednessN-;0<d<g-1}
If $0<d<g-1$, then the moduli space $\N_0(d)$ is connected.
\end{corollary}

Recall that we want to study the connected components of the moduli space $\N_0(d)$ of $0$-polystable quadric bundles, for any $0<d<2g-2$. From Corollaries \ref{cor:connectednessN-;g-1<d<2g-2} and \ref{cor:connectednessN-;0<d<g-1} we see that the only remaining case to understand is when $d=g-1$. Notice that the space $\N_0(g-1)$ is really $\N_{\alpha_m}(g-1)$.
Now, although the codimensions of every $\cS_{\alpha_k^\pm}(d)$ are only known under the condition $d<g-1$, it follows from Corollaries 3.11 and 3.15 of \cite{gothen-oliveira:2012} that the codimensions of both $\cS_{\alpha_m^\pm}(d)$ are known to be positive also when $d=g-1$. From this, arguing as in the last paragraph of the proof of \cite[Theorem 5.3]{gothen-oliveira:2012}, we prove that also $\N_0(g-1)$ is connected. So we conclude that:

\begin{theorem}\label{prop:N0connected}
The moduli space of $0$-polystable quadric bundles of degree $d$ is connected for every $0<d<2g-2$.
\end{theorem}

\section{Conclusion and further remarks}

\subsection{Non-maximal components for $\Sp(4,\R)$ and $\SO_0(2,3)$}

Proposition \ref{prop:boundcomp} and Theorem \ref{prop:N0connected} imply then that we have achieved our objective of calculating the number of connected components of  the moduli space of $\Sp(4,\R)$-Higgs bundles over $X$, with non-maximal and non-zero Toledo invariant:

\begin{theorem}\label{thm:connnectedSp4}
If $0<|d|<2g-2$ then  $\M_d(\Sp(4,\R))$ is connected.
\end{theorem}

This theorem has in fact been proved before \cite{gothen-oliveira:2012}, by O. Garc\'\i a-Prada and I. Mundet i Riera in \cite{garcia-prada-mundet:2004}, using different techniques. More precisely, they do consider quadric bundles, but prove the connectedness directly, i.e., fixing $\alpha=0$ and not implementing  the variation of the parameter.

\paragraph{}
Our method easily generalises for $U$-quadric bundles --- see Remark \ref{rmk:U-quad} --- for any line bundle $U$. In particular, if we consider $LK$-quadric bundles, with $L$ some line bundle of degree $1$, then these are related with $\SO_0(2,3)$-Higgs bundles with non-maximal (and non-zero) Toledo invariant, in the same way $K$-quadric bundles arise in the $\Sp(4,\R)$ case.
Applying Definition \ref{definition of Higgs bundle}, it is easy to check that an $\SO_0(2,3)$-Higgs bundle is defined by a tuple $(L,W,Q_W,\beta,\gamma)$ where $L$ is a line bundle, $(W,Q_W)$ is an orthogonal rank $3$ bundle and the Higgs field is defined by maps $\beta:W\to LK$ and $\gamma:W\to L^{-1}K$. These are topologically classified by the degree $d$ of $L$ (this is the Toledo invariant) and by the second Stiefel-Whitney class $w_2$ of $W$. 
Note that $\SO_0(2,3)$ is isomorphic to $\mathrm{PSp}(4,\R)$. Moreover, an $\SO_0(2,3)$-Higgs bundle lifts to an $\Sp(4,\R)$-Higgs bundle if and only if $d\equiv w_2\ \mathrm{mod}\ 2$. 

The precise same methods that we described for $\Sp(4,\R)$, yield then the following (see \cite[Theorem 6.26]{gothen-oliveira:2012}):

\begin{theorem}\label{thm:connnectedSO(2,3)}
If $0<|d|<2g-2$ and $w_2\in\Z/2$ then $\M_{d,w_2}(\SO_0(2, 3))$ is connected.
\end{theorem}

Recalling that non-abelian Hodge theory implies that the moduli space of $G$-Higgs bundles over $X$ is homeomorphic to the space of conjugacy classes of reductive representations of $\pi_1(X)$ in $G$ (cf. Remark \ref{rmk:NonabelianHodge}), we conclude the both Theorems \ref{thm:connnectedSp4} and \ref{thm:connnectedSO(2,3)} have their counterparts on the representations side (see \cite{garcia-prada-mundet:2004,gothen-oliveira:2012} for the detailed statements).

\subsection{Some different directions}

\subsubsection{Torelli theorem}
Our method to analyse the components of the moduli spaces $\N_\alpha(d)$ of $\alpha$-semistable quadric bundles of degree $d$ was to start with the study in the lowest extreme of $\alpha$, that is the study of $\N_{\alpha_m^-}(d)$. One can ask what happens in the highest possible extreme, namely $\alpha=\alpha_M=d/2$. Since this is a critical value, take a slight lower value, $\alpha_M^-=d/2-\epsilon$, for a small $\epsilon>0$. Here, different phenomena arise.

Briefly, it is easy to check that if $(V,\gamma)$ is $\alpha_M^-$-semistable, then $V$ is itself semistable as a rank $2$ vector bundle. If $M(d)$ denotes the moduli space of polystable rank $2$ degree $d$ vector bundles on $X$, this yields a forgetful map $\pi:\N_{\alpha_M^-}(d)\to M(d)$. If $d=2g-2$, this map is an embedding because $\N_{\alpha_M^-}(d)$ is, as we saw, the moduli of orthogonal vector bundles and thus follows from \cite{serman:2008}. If $g-1\leq d< 2g-2$, the determination of the image of $\pi$ is a Brill-Noether problem. If $d<g-1$, $\pi$ is surjective and if, further, $d<0$, the map $\pi$, suitably restricted, is a projective bundle over the stable locus $M^s(d)\subset M(d)$. This is explained in Proposition 3.13 of \cite{gothen-oliveira:2012}.

So assume $d<0$, and from now on let us just consider quadric bundles $(V,\gamma)$ where the determinant of $V$ is fixed to be some line bundle $\Lambda$ of degree $d$. Let $\N_\alpha(\Lambda)\subset\N_\alpha(d)$ and $M(\Lambda)\subset M(d)$ denote the corresponding obvious moduli spaces. Using the projective bundle $\pi$ onto the stable locus of $M(\Lambda)$ and through a detailed analysis of the the smooth locus $\N^{sm}_\alpha(\Lambda)\subset\N_\alpha(\Lambda)$, we were able to obtain some geometric and topological results on $\N_\alpha(d)$. This procedure is taken in detail in \cite{oliveira:2015} again in the more general setting of $U$-quadric bundles.

For instance we proved that $\N_\alpha(\Lambda)$ is irreducible and $\N^{sm}_\alpha(\Lambda)$ is simply-connected --- see Corollaries 4.3 and 4.4 of \cite{oliveira:2015}. The irreducibility was already known from \cite{gomez-sols:2000}, using different methods.

Under some slight conditions on the genus of $X$, we calculated the torsion-free part of the first three integral cohomology groups of the smooth locus $\N^{sm}_\alpha(\Lambda)\subset\N_\alpha(\Lambda)$ for any $\alpha$. In particular \cite[Proposition 5.6]{oliveira:2015} says that $H^3(\N^{sm}_\alpha(\Lambda),\Z)$ is isomorphic to $H^1(X,\Z)$. This fact, together with the assumption that the genus of $X$ is at least $5$, and after properly defining a polarisation on $H^3(\N^{sm}_\alpha(\Lambda),\Z)$ compatible with the one on $H^1(X,\Z)$, allowed us to prove that a Torelli type theorem holds for $\N_\alpha(\Lambda)$. From this it follows that the same is also true for the non-fixed determinant moduli. To emphasise now the base curve, write $\N_{X,\alpha}(\Lambda)$ for the moduli space of $\alpha$-polystable quadric bundles of rank two with fixed determinant $\Lambda$ on $X$. Let $\N_{X,\alpha}(d)$ be the same thing but just fixing the degree and not the determinant.

\begin{theorem}[\cite{oliveira:2015}]
Let $X$ and $X'$ be smooth projective curves of genus $g,g'\geq 5$, $\Lambda$ and $\Lambda'$ line bundles of degree $d<0$ and $d'<0$ on $X$ and $X'$, respectively. If $\N_{X,\alpha}(\Lambda)\cong \N_{X',\alpha}(\Lambda')$ then $X\cong X'$. 
The same holds for $\N_{X,\alpha}(d)$ and $\N_{X',\alpha}(d')$.
\end{theorem}

In other words, the isomorphism class of the curve $X$ is determined by the one of the projective variety $\N_{X,\alpha}(\Lambda)$.

\subsubsection{Higher ranks}

One natural question is to wonder if the procedure we described here can be generalised to ranks higher than $2$. First, Proposition \ref{prop:boundcomp} is true for any rank (it is even true for any real reductive Lie group). Proposition \ref{prop:localmin} also generalises in a straightforward way for $\Sp(2n,\R)$ for $n>2$, so we are again lead to the study of higher rank quadric bundles. The technical problems start here because the $\alpha$-semistability condition can be much more complicated in higher rank, involving not only subbundles but filtrations (see \cite{gomez-sols:2000} and \cite{giudice-pustetto:2014}). One consequence is that the study of $\N_{\alpha_m^-}(d)$ and mainly of $\cS_{\alpha_k^\pm}(d)$ should become much more complicated.



\footnotesize{\textbf{Acknowledgments:}

Author partially supported by CMUP (UID/MAT/00144/2013), by the Projects EXCL/MAT-GEO/0222/2012 and PTDC/MAT­GEO/2823/2014 and also by the Post-Doctoral fellowship SFRH/BPD/100996/2014. These are funded by FCT (Portugal) with national (MEC) and European structural funds (FEDER), under the partnership agreement PT2020. Support from U.S. National Science Foundation grants DMS 1107452, 1107263, 1107367 ``RNMS: GEometric structures And Representation varieties'' (the GEAR Network) is also acknowledged.

The author would like to thank Martin Schlichenmaier for the invitation to write this article, partially based on the talk given at the International Conference on Geometry and Quantization (GEOQUANT 2015), held at the ICMAT (Madrid). Acknowledgments also due to the organisers of the Conference.

The main content of this article is based on joint work with Peter Gothen to whom the author is greatly indebted.}
\vspace{.5cm}

\normalsize

\noindent
      \textbf{Andr\'e Oliveira} \\
      Centro de Matem\'atica da Universidade do Porto, CMUP\\
      Faculdade de Ci\^encias, Universidade do Porto\\
      Rua do Campo Alegre 687, 4169-007 Porto, Portugal\\ 
      email: andre.oliveira@fc.up.pt

\vspace{.2cm}
\noindent
\textit{On leave from:}\\
 Departamento de Matem\'atica, Universidade de Tr\'as-os-Montes e Alto Douro, UTAD \\
Quinta dos Prados, 5000-911 Vila Real, Portugal\\ 
email: agoliv@utad.pt


\label{lastpage}
\end{document}